\documentclass{amsart}
\usepackage{pdfpages}
\usepackage{hyperref}
\newtheorem{theorem}{Theorem}[section]

\newtheorem{corollary}[theorem]{Corollary}
\theoremstyle{definition}
\newtheorem{definition}[theorem]{Definition}

\theoremstyle{remark}

\numberwithin{equation}{section}

\begin{document}

\title{Rectifying curves under conformal transformation}

\author{Absos Ali Shaikh$^1$, Mohamd Saleem Lone$^2$ and Pinaki Ranjan Ghosh$^{3}$}
\address{$^1$Department of Mathematics, University of
Burdwan, Golapbag, Burdwan-713104, West Bengal, India}
\email{aask2003@yahoo.co.in, aashaikh@math.buruniv.ac.in}
\address{$^2$International Centre for Theoretical Sciences, Tata Institute of Fundamental Research, 560089, Bengaluru, India}
\email{saleemraja2008@gmail.com, mohamdsaleem.lone@icts.res.in}
\address{$^3$Department of Mathematics, University of
Burdwan,Golapbag, Burdwan-713104, West Bengal, India}
\email{mailtopinaki94@gmail.com}






\begin{abstract}
The main aim of this paper is to investigate the nature of invariancy of rectifying curve under conformal transformation and obtain a sufficient condition for which such a curve remains conformally invariant. It is shown that the normal component and the geodesic curvature of the rectifying curve is homothetic invariant.
\end{abstract}

\maketitle
\section{Introduction} 
In geometry one of the most important field is the study of the differential geometric properties of smooth maps between surfaces(manifolds). Two surfaces ${\bf M}$ and ${\overline{\bf M}}$ are said to be mapped onto one another if there is a one-to-one correspondence between their points. Out of the many, the interesting one's are those which preserves certain geometric properties. With respect to fundamental forms, mean curvature $(H)$ and the Gaussian curvatures $(K)$, we broadly classify the motions(transformations) as isometric, conformal and non-conformal(general motion). Isometry preserves length as well as the angle between the curves on the surface. Geometrically, isometry preserves $K$ but not $H$. A best known example of such an isometry is of a plane and a cylinder. The most important type of transformations is conformal, where only angles are preserved both in magnitude and orientation but not necessarily distances. One of the simplest example of conformal maps is stereographic projection. This property of conformal maps is believed to be first used by Gerardus Mercator to produce the famous Mercator's world map of 1569, the first angle-preserving (or conformal) world map. For animated explaination and application of the conformal maps, we refer to see \cite{1}. In case of general motions neither angles nor distances are preserved between any intersecting pair of curves on a surface. Throughout the paper by $\bf M$ and $\overline{\bf M}$ we mean surfaces immersed in $\mathbb{E}^3$ and all the geometric objects on $\overline{\bf M}$ will be denoted by the notation $'-'$ bar.

Let $I\subset\mathbb{R}$ be an interval and $\alpha:I\subset \mathbb{R}\rightarrow \mathbb{E}^3$ be a unit speed smooth curve. Let $\vec{t}$, $\vec{n}$ and $\vec{b}$ be respectively the unit tangent, normal and binormal vector to the curve $\alpha(s)$ at any point $\alpha(s)$ such that $\{ \vec{t},\vec{n},\vec{b}\}$ acts as its Serret-Frenet frame. Then, the Serret-Frenet formulae are given by 
\begin{eqnarray*}
\left\{
\begin{array}{ll}
\vec{t}^\prime =\kappa \vec{n}\\
\vec{n}^\prime = -\kappa \vec{t} +\tau \vec{b}\\
\vec{b}^\prime =-\tau \vec{n},
\end{array}
\right.
\end{eqnarray*}
where $\kappa$ is the curvature and $\tau$ is the torsion of $\alpha$ with $\vec{t}=\alpha^\prime, \vec{n}=\frac{\vec{t}^\prime}{\kappa}$ and $\vec{b}=\vec{t}\times \vec{n}$, and $'\prime\ '$ denotes $\frac{d}{ds}$.
 At every point of $\alpha(s)$ the planes spanned by $\{\vec{n},\vec{b}\}$, $\{\vec{t},\vec{n}\}$ and $\{\vec{t},\vec{b}\}$ are respectively called the normal plane, the osculating plane and the rectifying plane. Also if at each point the position vector of $\alpha$ lies in the osculating plane (respectively, normal plane), then the curve lies in a plane (respectively, on a sphere). In 2003 Chen \cite{2} introduced the notion of rectifying curves and obtained their characterization. For generic study, we refer the reader to see \cite{3,4}.

In 2018 Shaikh and Ghosh \cite{8} studied invariancy of rectifying curves under surface isometry and showed that the normal component of such a curve is invariant under isometry. Also Shaikh and Ghosh \cite{aas8B} investigated the invariancy of osculating curves under surface isometry. Again in \cite{9} Shaikh {\it et al.} studied normal curves under isometric motion. Motivated by the above studies we investigate the nature of rectifying curves under conformal transformation and provide the answer of the following question.

{\it Question:} What happens to a rectifying curve on a smooth immersed surface with respect to a conformal transformation?

 We obtain a sufficient condition for a rectifying curve on a surface to be invariant under conformal map. It is shown that the normal component and the geodesic curvature of a rectifying curve are homothetic invariant.
The structure of this paper is as follows. Section $2$ is devoted to some rudimentary facts about the curves lying on a surface. In section $3$, we discuss the main results.

\section{Preliminaries}
\begin{definition} A diffeomorphism $J:{\bf M}\rightarrow \overline{{\bf M}}$ is called an isometry if for all $p \in {\bf M }$ and $x_1,x_2 \in T_p{\bf M}$, the following holds: $$\langle x_1, x_2 \rangle_p =\langle dJ_p(x_1),dJ_p(x_2) \rangle_{J(p)},$$ $T_p\bf M$ being the tangent plane at $p\in M$.
In this case the surfaces ${\bf M}$ and $\overline{{\bf M}}$ are said to be isometric.
\end{definition}
\begin{definition}
A diffeomorphism $J:V \subset {\bf M} \rightarrow \overline{{\bf M}}$ of a neighborhood $V$ of $p \in {\bf M}$ is called a local isometry at $p$ if there exists a neighborhood $U$ of $J(p) \in \overline{\bf M}$ such that $J: V \rightarrow U$ is an isometry. If there exists local isometry at every point of ${\bf M}$, then ${\bf M}$ and $\overline{{\bf M}}$ are said to be locally isometric. Clearly, if $J$ is a local isometry at every point of ${\bf M}$, then $J$ is called a global isometry.
\end{definition}
It is well-known that the coefficients of the first fundamental form of a surface are invariant under isometry. Hence if $E,F,G$ and $\overline{E},\overline{F},\overline{G}$ are the coefficients of first fundamental form of ${\bf M}$ and $\overline{{\bf M}}$, respectively and $J:{\bf M}\rightarrow \overline{{\bf M}}$ is a local isometry, then 
\begin{equation*}
E=\overline{E},\quad F=\overline{F},\quad G=\overline{G}.
\end{equation*}
\begin{definition} A diffeomorphism $J:{\bf M}\rightarrow \overline{{\bf M}}$ is called a conformal map if for all $p \in {\bf M }$ and $x_1,x_2 \in T_p{\bf M}$, 
$$\lambda^2\langle dJ_p(x_1),dJ_p(x_2) \rangle_{J(p)}=\langle x_1, x_2 \rangle_p $$ holds, where $\lambda^2$ is a differentiable function on ${\bf M}$ and is sometimes called as dilation or scale factor. If such a diffeomorphism exists for each $p \in {\bf M}$, then ${\bf M}$ and $\overline{{\bf M}}$ are said to be conformal(locally).  Thus a conformal transformation is the composition of an isometric transformation and a dilation and if the dilation factor is identity, then it coincides with isometry. Geometrically, conformal maps preserve angles but not necessarily the lengths. In this case \cite{5}, we have 
\begin{equation*}
\lambda^2E=\overline{E},\quad  \lambda^2F=\overline{F},\quad  \lambda^2G=\overline{G},
\end{equation*}
and we call that the coefficients of the first fundamental form are conformally invariant.
\end{definition}
A necessary and sufficient condition for $J$ to be conformal is that the area elements of arcs on ${\bf M}$ and $\overline{\bf M}$ are proportional and the ratio is equal to the dilation factor, i.e., $\frac{ds}{d\overline{s}}=\lambda(u,v)$. If the dilation factor is a non-zero constant(say $c$) for all the points of the surface, then the conformal map is called homothetic. The conformal map reduces to an isometry when the dilation factor is equal to $1$. In other words isometric maps can be considered as a subset of conformal maps with the dilation factor $\lambda=1$ \cite{6}.

\begin{definition}\label{def1}
Let $f:{\bf M}\rightarrow \overline{\bf M}$ be a conformal(or homothetic) map between two smooth surfaces and $\overline{f}=f\circ\varphi$, where $\varphi$ is a surface patch of $M$, then we say that $f$ is conformally(or homothetic) invariant if $\overline{f}=\lambda^2f($ or$\overline{f}=c^2f)$ for some dilation factor $\lambda(\text{or }c\neq \{0,1\})$. 
\end{definition}

\begin{definition}
 A curve $\alpha:I\rightarrow\mathbb{E}^3,\ I\subset\mathbb{R},$ is said to be a rectifying curve if its position vector lies in the orthogonal complement of normal vector $\vec{n}$ i.e., $\alpha \cdot \vec{n} =0,$ or
\begin{equation}\label{1}
\alpha(s)=\xi(s)\vec{t}(s)+ \mu(s)\vec{b}(s),
\end{equation}
where $\xi,$ $\mu$ are smooth functions.
\end{definition}
Suppose ${\bf M}$ is a regular surface([page no 52, \cite{5}]) with 
$\varphi(u,v)$ being its coordinate chart. Then, the curve 
$\alpha(s)=\varphi(u(s),v(s))$ defines a curve 
on the surface ${\bf M}$. We can easily find the derivatives of the curve $\alpha(s)$ as a curve on the surface ${\bf M}$ using the chain rule:
\begin{eqnarray}
\nonumber\alpha^\prime(s)&=&\varphi_uu^\prime+\varphi_vv^\prime\\
\nonumber\text{or }&&\\
\label{2} \vec{t}(s)&=&\alpha^\prime(s)=\varphi_uu^\prime+\varphi_vv^\prime\\
\nonumber
 {\vec{t}\ '}(s)&=& u^{\prime\prime}\varphi_u+v^{\prime\prime}\varphi_v+{u^\prime}^2\varphi_{uu}+2u^\prime v^\prime \varphi_{uv}+{v^\prime}^2\varphi_{vv}.
\end{eqnarray}
If ${\bf N}$ is the surface normal, then we have
\begin{eqnarray}\label{3}
\nonumber
\vec{n}(s)&=&\frac{1}{k(s)}(u''\varphi_u+v''\varphi_v+u'^2\varphi_{uu}+2u'v'\varphi_{uv}+v'^2\varphi_{vv}).\\
\nonumber
\vec{b}(s)&=& \vec{t}(s)\times \vec{n}(s)= \vec{t}(s)\times \frac{{\vec{t}}^\prime(s)}{k(s)},\\
\nonumber
&=&\frac{1}{k(s)}\Big[(\varphi_uu'+\varphi_vv')\times(u''\varphi_u+v''\varphi_v+u'^2\varphi_{uu}+2u'v'\varphi_{uv}+v'^2\varphi_{vv})\Big],\\
\nonumber&=&\frac{1}{k(s)}\Big[\{u'v''-u''v'\}{\bf N}+u'^3\varphi_u\times \varphi_{uu}+2u'^2v'\varphi_u\times \varphi_{uv}+u'v'^2\varphi_u\times \varphi_{vv}\\
&&+u'^2v'\varphi_v\times \varphi_{uu}+2u'v'^2\varphi_v\times \varphi_{uv}+v'^3\varphi_v\times \varphi_{vv}\Big].
\end{eqnarray}
\par
Since $\alpha(s)$ is an unit speed curve on the surface, so $\alpha''\perp\alpha'$ and hence $\alpha''$ lies in the plane spanned by $N$ and $\alpha'\times N$, i.e.,
\begin{equation*}
\alpha^{\prime\prime}=\kappa_{n}{\bf N}+\kappa_g{\bf N}\times \alpha^\prime,
\end{equation*}
where $\kappa_n$ and $\kappa_g$ are respectively called the normal curvature and the geodesic curvature of $\alpha$. 
Since $\alpha^{\prime \prime}=\kappa(s)\vec{n}(s)$, we can write
\begin{equation*}
\kappa_n=\kappa(s)\vec{n}(s)\cdot {\bf N}=(u''\varphi_u+v''\varphi_v+u'^2\varphi_{uu}+2u'v'\varphi_{uv}+v'^2\varphi_{vv})\cdot{\bf N}
\end{equation*}
or 
\begin{equation}\label{se1}
\kappa_n = {u^\prime}^2 L + 2u^\prime v^\prime M +{v^\prime}^2N,
\end{equation}
where $L,M,N$ are the coefficients of the second fundamental form of the surface. The curve $\alpha$ on $\bf{M}$ is said to be asymptotic if and only if $\kappa_n=0.$

\section{Conformal image of a rectifying curve}
Let $\alpha(s)$ be a rectifying curve lying on a smooth immersed surface ${\bf M}$ in $\mathbb{E}^3$. Then by virtue of (\ref{1}), (\ref{2}) and (\ref{3}), we obtain
\begin{eqnarray}\label{2.1}
\nonumber
\alpha(s)&=&\xi(s)(\varphi_uu'+\varphi_vv')+\frac{\mu(s)}{k(s)}\Big[\{u'v''-u''v'\}{\bf N}+u'^3\varphi_u\times \varphi_{uu}+2u'^2v'\varphi_u\times \varphi_{uv}\\
&&+u'v'^2\varphi_u\times \varphi_{vv}+u'^2v'\varphi_v\times \varphi_{uu}+2u'v'^2\varphi_v\times \varphi_{uv}+v'^3\varphi_v\times \varphi_{vv}\Big].
\end{eqnarray}
In the following theorem we consider the expression $J_*(\alpha(s))$ as a product of a $3\times 3$ matrix $J_*$ and a $3\times 1$ matrix $\alpha(s)$.

\begin{theorem}
Let $J:{\bf M}\rightarrow {\overline{\bf M}}$ be a conformal map between two smooth immersed surfaces ${\bf M}$ and $\overline{\bf M}$ in $\mathbb{E}^3$ and $\alpha(s)$ be a rectifying curve on ${\bf M}$. Then $\overline{\alpha}(s)$ is a rectifying curve on $\overline{\bf M}$, if
\begin{eqnarray}\label{j2.2}
\nonumber
\overline{\alpha}&=&\frac{\mu}{\kappa}\Big[{u^\prime}^3 \lambda J_\ast \varphi_u \times \left(\lambda_u J_\ast \varphi_u +  \lambda\frac{\partial J_*}{\partial u}\varphi_u\right)+2{u^\prime}^2 v^\prime \lambda J_\ast \varphi_u \times \left(\lambda_v J_\ast \varphi_u +  \lambda\frac{\partial J_*}{\partial v}\varphi_u\right)\\
\nonumber &&+u^\prime {v^\prime}^2\lambda J_\ast \varphi_u \times \left(\lambda_v J_\ast \varphi_v +  \lambda\frac{\partial J_*}{\partial v}\varphi_v\right)+{u^\prime}^2v^\prime \lambda J_\ast \varphi_v \times \left(\lambda_u J_\ast \varphi_u +  \lambda\frac{\partial J_*}{\partial u}\varphi_u\right)\\
\nonumber &&+2u^\prime {v^\prime}^2\lambda J_\ast \varphi_v \times \left(\lambda_u J_\ast \varphi_v +  \lambda\frac{\partial J_*}{\partial u}\varphi_v\right)+{v^\prime}^3\lambda J_\ast \varphi_v \times \left(\lambda_v J_\ast \varphi_v +  \lambda\frac{\partial J_*}{\partial v}\varphi_v\right)\Big]\\&&+\lambda J_\ast(\alpha).
\end{eqnarray}
\end{theorem}
\begin{proof}
Let $\overline{\bf M}$ be the conformal image of ${\bf M}$ and $\varphi(u,v)$ and $\overline{\varphi}(u,v)=J\circ \varphi(u,v) $ be the surface patches of ${\bf M}$ and $\overline{{\bf M}},$ respectively. Then the differential map $dJ=J_\ast$ of $J$ sends each vector of the tangent plane $T_p{\bf M}$ to a dilated tangent vector of the tangent plane of $T_{J(p)}\overline{\bf M}$ with the dilation factor $\lambda$. Also
\begin{eqnarray}\label{2.2}
\overline{\varphi}_u(u,v)&=&\lambda(u,v) J_*(\varphi(u,v))\varphi_u,\\
\label{2.3}\overline{\varphi}_v(u,v)&=&\lambda(u,v) J_*(\varphi(u,v))\varphi_v.
\end{eqnarray}
Differentiating $(\ref{2.2})$ and $(\ref{2.3})$ partially with respect to both $u$ and $v$ respectively, we get
\begin{eqnarray}\label{q3.7}
\nonumber \overline{\varphi}_{uu}&=& \lambda_u J_\ast \varphi_u +  \lambda\frac{\partial J_*}{\partial u}\varphi_u+\lambda J_*\varphi_{uu}\\
\overline{\varphi}_{vv}&=&\lambda_v J_\ast \varphi_v+ \lambda\frac{\partial J_*}{\partial v}\varphi_v+\lambda J_*\varphi_{vv}\\
\nonumber \overline{\varphi}_{uv}&=& \lambda_u J_\ast \varphi_v+ \lambda\frac{\partial J_*}{\partial u}\varphi_v+\lambda J_*\varphi_{uv}\\
\nonumber &=& \lambda_v J_\ast \varphi_u+ \lambda\frac{\partial J_*}{\partial v}\varphi_u +\lambda J_*\varphi_{uv}.\\\nonumber
 \end{eqnarray}
We can write
\begin{eqnarray}\label{j2.6}
\nonumber \lambda J_\ast \varphi_u \times \left(\lambda_u J_\ast \varphi_u+\lambda\frac{\partial J_\ast}{\partial u}\varphi_u\right)
 &=&\lambda J_\ast \varphi_u \times \left(\lambda_u J_\ast \varphi_u+\lambda\frac{\partial J_\ast}{\partial u}\varphi_u+\lambda J_\ast \varphi_{uu}\right)\\
\nonumber &&-\lambda J_\ast (\varphi_u \times \varphi_{uu})\\
&=&\overline{\varphi}_u \times \overline{\varphi}_{uu}-\lambda J_\ast(\varphi_u \times \varphi_{uu}).
\end{eqnarray} 
Similarly
\begin{eqnarray}\label{j2.7}
\left\{
\begin{array}{ll}
\lambda J_\ast \varphi_u \times \left(\lambda_v J_\ast \varphi_u+\lambda\frac{\partial J_\ast}{\partial v}\varphi_u\right)=\overline{\varphi}_u \times \overline{\varphi}_{uv}-\lambda J_\ast(\varphi_u \times \varphi_{uv})\vspace{.2cm}\\
\lambda J_\ast \varphi_u \times \left(\lambda_v J_\ast \varphi_v+\lambda\frac{\partial J_\ast}{\partial v}\varphi_v\right)=\overline{\varphi}_u \times \overline{\varphi}_{vv}-\lambda J_\ast(\varphi_u \times \varphi_{vv})\vspace{.2cm}\\
\lambda J_\ast \varphi_v \times \left(\lambda_u J_\ast \varphi_u+\lambda\frac{\partial J_\ast}{\partial u}\varphi_u\right)=\overline{\varphi}_v \times \overline{\varphi}_{uu}-\lambda J_\ast(\varphi_v \times \varphi_{uu})\vspace{.2cm}\\
\lambda J_\ast \varphi_v \times \left(\lambda_u J_\ast \varphi_v+\lambda\frac{\partial J_\ast}{\partial u}\varphi_v\right)=\overline{\varphi}_v \times \overline{\varphi}_{uv}-\lambda J_\ast(\varphi_v \times \varphi_{uv})\vspace{.2cm}\\
\lambda J_\ast \varphi_v \times \left(\lambda_v J_\ast \varphi_v+\lambda\frac{\partial J_\ast}{\partial v}\varphi_v\right)=\overline{\varphi}_v \times \overline{\varphi}_{vv}-\lambda J_\ast(\varphi_v \times \varphi_{vv})\vspace{.2cm}.\\
\end{array}
\right.
\end{eqnarray} 
 Therefore in view of (\ref{j2.2}), (\ref{j2.6}) and (\ref{j2.7}), we have
\begin{eqnarray*}
\nonumber
\overline{\alpha}&=&\xi(s)(\lambda J_\ast \varphi_uu'+\lambda J_\ast \varphi_vv')+\frac{\mu}{\kappa}\Big[\{u'v''-u''v'\}\lambda^2 J_\ast {\bf N}+u'^3\lambda J_\ast (\varphi_u\times \varphi_{uu})\\
&&+2u'^2v'\lambda J_\ast ( \varphi_u\times \varphi_{uv})
+u'v'^2\lambda J_\ast (\varphi_u\times \varphi_{vv})+u'^2v'\lambda J_\ast (\varphi_v\times \varphi_{uu})\\
&&+2u'v'^2\lambda J_\ast (\varphi_v\times \varphi_{uv})+v'^3\lambda J_\ast( \varphi_v\times \varphi_{vv})\Big]+\frac{\mu}{\kappa}\Big[{u^\prime}^3 \lambda J_\ast \varphi_u \times \left(\lambda_u J_\ast \varphi_u +  \lambda\frac{\partial J_*}{\partial u}\varphi_u\right)\\
&&+2{u^\prime}^2 v^\prime \lambda J_\ast \varphi_u \times \left(\lambda_v J_\ast \varphi_u +  \lambda\frac{\partial J_*}{\partial v}\varphi_u\right)
+u^\prime {v^\prime}^2\lambda J_\ast \varphi_u \times \left(\lambda_v J_\ast \varphi_v +  \lambda\frac{\partial J_*}{\partial v}\varphi_v\right)\\
&&+{u^\prime}^2v^\prime \lambda J_\ast \varphi_v \times \left(\lambda_u J_\ast \varphi_u +  \lambda\frac{\partial J_*}{\partial u}\varphi_u\right)
+2u^\prime {v^\prime}^2\lambda J_\ast \varphi_v \times \left(\lambda_u J_\ast \varphi_v +  \lambda\frac{\partial J_*}{\partial u}\varphi_v\right)\\
&&+{v^\prime}^3\lambda J_\ast \varphi_v \times \left(\lambda_v J_\ast \varphi_v +  \lambda\frac{\partial J_*}{\partial v}\varphi_v\right)\Big],
\end{eqnarray*}
which can be written as
\begin{eqnarray*}
\nonumber
\overline{\alpha}(s)&=&\xi(s)(\overline{\varphi}_uu'+\overline{\varphi}_vv')+\frac{\mu(s)}{k(s)}\Big[\{u'v''-u''v'\}\overline{{\bf N}}+u'^3\overline{\varphi}_u\times \overline{\varphi}_{uu}+2u'^2v'\overline{\varphi}_u\times \overline{\varphi}_{uv}\\
&&+u'v'^2\overline{\varphi}_u\times \overline{\varphi}_{vv}+u'^2v'\overline{\varphi}_v\times \overline{\varphi}_{uu}+2u'v'^2\overline{\varphi}_v\times \overline{\varphi}_{uv}+v'^3\overline{\varphi}_v\times \overline{\varphi}_{vv}\Big],
\end{eqnarray*}
or,
\begin{equation*}
\overline{\alpha}(s)=\overline{\xi}(s)\overline{\vec{t}}(s)+\frac{\overline{\mu}(s)}{\overline{\kappa}(s)}\overline{\vec{b}}(s),
\end{equation*}
for some smooth functions $\overline{\xi}(s)$ and $\overline{\mu}(s).$ Here and now onward, we assume $\overline{\xi} \approx  \xi$ and $\frac{\overline{\mu}(s)}{\overline{\kappa}(s)} \approx \frac{\mu}{\kappa}.$ Thus $\overline{\alpha}(s)$ is a rectifying curve.
\end{proof}
\begin{corollary}
Let $J:{\bf M}\rightarrow \overline{\bf M}$ be a homothetic map, where ${\bf M}$ and $\overline{\bf M}$ are smooth surfaces and $\alpha(s)$ be a rectifying curve on ${\bf M}$. Then $\overline{\alpha}(s)$ is a rectifying curve on $\overline{\bf M}$ if
\begin{eqnarray*}
\nonumber
\overline{\alpha}&=&\frac{\mu}{\kappa}\Big[{u^\prime}^3 c \left( J_\ast \varphi_u \times \frac{\partial J_*}{\partial u}\varphi_u\right)+2{u^\prime}^2 v^\prime c \left(J_\ast \varphi_u \times   c\frac{\partial J_*}{\partial v}\varphi_u\right)\\
\nonumber &&+u^\prime {v^\prime}^2c\left( J_\ast \varphi_u \times   c\frac{\partial J_*}{\partial v}\varphi_v\right)+{u^\prime}^2v^\prime c \left(J_\ast \varphi_v \times   c\frac{\partial J_*}{\partial u}\varphi_u\right)\\
\nonumber &&+2u^\prime {v^\prime}^2c \left(J_\ast \varphi_v \times  c\frac{\partial J_*}{\partial u}\varphi_v\right)+{v^\prime}^3c\left(J_\ast \varphi_v \times  c\frac{\partial J_*}{\partial v}\varphi_v\right)\Big]\\&&+c J_\ast(\alpha).
\end{eqnarray*}
\end{corollary}
\begin{proof}
In case of a homothetic map the dilation function $\lambda=c \neq \{0,1\}$. Substituting in (\ref{j2.2}), we get the above expression.  
\end{proof}
\begin{corollary}\cite{8}
Let $J:{\bf M}\rightarrow \overline{\bf M}$ be an isometry, where ${\bf M}$ and $\overline{\bf M}$ are smooth surfaces and $\alpha(s)$ be a rectifying curve on ${\bf M}$. Then $\overline{\alpha}(s)$ is a rectifying curve on $\overline{\bf M}$ if
\begin{eqnarray*}
\overline{\alpha}(s)-J_*(\alpha(s))&=&\frac{\mu(s)}{k(s)}\Big[u'^3\Big(J_*\varphi_u\times \frac{\partial J_*}{\partial u}\varphi_u\Big)+2u'^2v'\Big(J_*\varphi_u\times \frac{\partial J_*}{\partial u}\varphi_v\Big)\\
&&+u'v'^2\Big(J_*\varphi_u\times \frac{\partial J_*}{\partial v}\varphi_v\Big)+u'^2v'\Big(J_*\varphi_v\times \frac{\partial J_*}{\partial u}\varphi_u\Big)\\
&&+2u'v'^2\Big(J_*\varphi_v\times \frac{\partial J_*}{\partial u}\varphi_v\Big)+v'^3\Big(J_*\varphi_v\times \frac{\partial J_*}{\partial v}\varphi_v\Big)\Big].
\end{eqnarray*}
\end{corollary}
\begin{proof}
A conformal transformation is the composition of a dilation function and an isometry. Substituting $\lambda=1$ in (\ref{j2.2}), we get the above expression.  
\end{proof}
\begin{theorem}
Let ${\bf M}$ and $\overline{\bf M}$ be two conformal smooth surfaces and $\alpha(s)$ be a rectifying curve on ${\bf M}$. Then for the component of $\alpha(s)$ along the surface normal, we have
\begin{equation}\label{q8}
\overline{\alpha} \cdot \overline{\bf N}= \lambda^4 [\alpha \cdot {\bf N}+ h(E,G,F,\lambda)],
\end{equation}
where 
\begin{equation}\label{q9}
h(E,G,F,\lambda)= \frac{\mu}{\kappa}W^2\Big[{u^\prime}^3\varepsilon_{11}^2-{v^\prime}^3 \varepsilon_{22}^1 +2 {u^\prime}^2 v^\prime \varepsilon_{12}^2 +u^\prime {v^\prime}^2 \varepsilon_{22}^2+({u^\prime}^2v^\prime +2u^\prime {v^\prime}^2)\varepsilon_{12}^1\Big].
\end{equation}
\end{theorem}
\begin{proof}
Let $\overline{\bf M}$ be the conformal image of ${\bf M}$ and $\varphi(u,v)$ and $\overline{\varphi}(u,v)=J\circ \varphi(u,v) $ be the surface patches of ${\bf M}$ and $\overline{{\bf M}},$ respectively. We know that 
\begin{equation}\label{z2.8}
\lambda^2E=\overline{E},\quad  \lambda^2F=\overline{F},\quad  \lambda^2G=\overline{G}.
\end{equation}
This implies that
\begin{eqnarray}\label{12a}
\left\{
\begin{array}{ll}
\overline{E}_u=2\lambda \lambda_u E + \lambda^2 E_v, \quad \overline{E}_v=2\lambda \lambda_v E + \lambda^2 E_v,\\
\overline{F}_u=2\lambda \lambda_u F + \lambda^2 F_v, \quad \overline{F}_v=2\lambda \lambda_v F + \lambda^2 F_v,\\
\overline{G}_u=2\lambda \lambda_u G + \lambda^2 G_v, \quad \overline{G}_v=2\lambda \lambda_v G + \lambda^2 G_v.
\end{array}
\right.
\end{eqnarray}
Now, we have
$$E_u=(\varphi_u \cdot \varphi_u)_u=2\varphi_{uu}\cdot \varphi_u$$
\begin{equation}\label{a2.9}
\varphi_{uu}\cdot \varphi_u=\frac{E_u}{2}.
\end{equation}
Similarly, it is easy to show that
\begin{eqnarray}\label{a2.10}
\begin{array}{ll}
\varphi_{uu}\cdot \varphi_v=F_u-\frac{E_v}{2}, \quad \varphi_{uv}\cdot \varphi_u=\frac{E_v}{2}, \quad \varphi_{uv}\cdot \varphi_v=\frac{G_u}{2},\\
\varphi_{vv}\cdot \varphi_v=\frac{G_v}{2}, \quad \varphi_{vv}\cdot \varphi_u=F_v-\frac{G_u}{2}
\end{array}
\end{eqnarray}
In addition, let $\Gamma_{ij}^k$ be the Christoffel symbols of second kind given by
\begin{equation}\label{q12}\left\{
\begin{array}{ll}
\Gamma_{11}^1=\frac{1}{2W^2}\left\{GE_u+F[E_v-2F_u]\right\}, \quad \Gamma_{22}^2=\frac{1}{2W^2}\left\{EG_v+F[G_v-2F_v]\right\}\\
\Gamma_{11}^2=\frac{1}{2W^2}\left\{E[2F_u-E_v]-FE_v\right\},\quad \Gamma_{22}^1=\frac{1}{2W^2}\left\{G[2F_v-G_u]-FG_v\right\}\\
\Gamma_{12}^2=\frac{1}{2W^2}\left\{EG_u-FE_v\right\}=\Gamma_{21}^2,\quad
\Gamma_{21}^1=\frac{1}{2W^2}\left\{GE_v-FG_u\right\}=\Gamma_{12}^1
\end{array}
\right.
\end{equation}
where $W=\sqrt{EG-F^2}$.
After conformal motion, the Christoffel symbols turns out to be
\begin{eqnarray}\label{q13}
\begin{array}{ll}
\overline{\Gamma}_{11}^1=\Gamma_{11}^1 +\varepsilon_{11}^1,\quad \overline{\Gamma}_{11}^2=\Gamma_{11}^2 +\varepsilon_{11}^2, \quad \overline{\Gamma}_{12}^1=\Gamma_{12}^1 +\varepsilon_{12}^1,\\
\overline{\Gamma}_{12}^2=\Gamma_{12}^2 +\varepsilon_{12}^2,\quad \overline{\Gamma}_{22}^1=\Gamma_{22}^1 +\varepsilon_{22}^1, \quad \overline{\Gamma}_{22}^2=\Gamma_{22}^2 +\varepsilon_{22}^2,
\end{array}
\end{eqnarray}
where
\begin{eqnarray}\label{q16}
\left\{\begin{array}{ll}
\varepsilon_{11}^1=
\frac{EG\lambda_u -2F^2\lambda_u +FE\lambda_v}{\lambda W^2},\quad \varepsilon_{11}^2=
\frac{EF\lambda_u -E^2 \lambda_v}{\lambda W^2},\vspace{.1cm}\\
\varepsilon_{12}^1=\frac{EG\lambda_v-FG\lambda_u}{\lambda W^2},
\quad \varepsilon_{12}^2=\frac{EG\lambda_u - FE \lambda_v}{\lambda W^2},\vspace{.1cm}\\
\varepsilon_{22}^1=\frac{GF\lambda_v - G^2 \lambda_u}{\lambda W^2},\quad \varepsilon_{22}^2=\frac{EG\lambda_v -2F^2 \lambda_v +FG\lambda_u}{\lambda W^2}.
\end{array}\right.
\end{eqnarray}
Now to find the position vector of the curve $\alpha$ along the normal ${\bf N}$ to the surface ${\bf M}$ at a point $\alpha(s)$, we have
\begin{eqnarray*}
\nonumber
\alpha(s)\cdot {\bf N} &=&\frac{\mu(s)}{k(s)}\Big[ (EG-F^2)(u^{\prime}v^{\prime \prime}-u^{\prime \prime}v)+(\varphi_u\times \varphi_{uu})\cdot{\bf N}{u^{\prime}}^3+2(\varphi_u\times \varphi_{uv})\cdot {\bf N}{u^{\prime}}^2v^{\prime}\\
\nonumber
&&+(\varphi_u\times \varphi_{vv})\cdot{\bf N}u^{\prime}{v^{\prime}}^2+(\varphi_v\times \varphi_{uu})\cdot{\bf N}{u^{\prime}}^2v^{\prime}+2(\varphi_v\times \varphi_{uv})\cdot{\bf N}u^{\prime}{v^{\prime}}^2\\
&&+(\varphi_v\times \varphi_{vv})\cdot{\bf N}{v^{\prime}}^3\Big],\\
&=&\frac{\mu(s)}{k(s)}\Big[(u^{\prime}v^{\prime \prime}-u^{\prime \prime}v)(EG-F^2)+{u^{\prime}}^3\{E(\varphi_{uu}\cdot\varphi_v)-F(\varphi_{uu}\cdot\varphi_u)\}\\
&&+2{u^{\prime}}^2v^{\prime}\{E(\varphi_{uv}\cdot\varphi_v)-F(\varphi_{uv}\cdot\varphi_u)\}+u^{\prime}{v^{\prime}}^2\{E(\varphi_{vv}\cdot\varphi_v)-F(\varphi_{vv}\cdot\varphi_u)\}\\
&&+{u^{\prime}}^2v^{\prime}\{F(\varphi_{uu}\cdot\varphi_v)-G(\varphi_{uu}\cdot\varphi_u)\}+2u^{\prime}{v^{\prime}}^2\{F(\varphi_{uv}\cdot\varphi_v)-G(\varphi_{uv}\cdot\varphi_u)\}\\
&&+{v^{\prime}}^3\{F(\varphi_{vv}\cdot\varphi_v)-G(\varphi_{vv}\cdot\varphi_u)\}.
\end{eqnarray*}
Using (\ref{a2.9}) and (\ref{a2.10}) in the above equation, we get
\begin{eqnarray*}
\alpha \cdot {\bf N} &=&\frac{\mu}{\kappa}\Big[(EG-F^2)(u^\prime v^{\prime\prime}-v^\prime u^{\prime\prime})
+\left \{E\left(F_u-\frac{E_v}{2}\right)-\frac{FE_u}{2}\right \}{u^\prime}^3\\
&&+2 \left \{\frac{EG_u}{2}-\frac{FE_v}{2}\right \}{u^\prime}^2 v^\prime
+\left \{\frac{EG_v}{2}-F\left(F_v-\frac{G_u}{2}\right) \right\}{u^\prime}{v^\prime}^2\\
&&+ \left \{\frac{FG_u}{2}-\frac{GE_u}{2}\right \}{u^\prime}^2v^\prime+2 \left \{\frac{FG_u}{2} - \frac{GE_v}{2} \right\}{u^\prime}{v^\prime}^2\\
&&+ \left \{\frac{FG_v}{2}-G\left(F_v-\frac{G_u}{2}\right)\right \}{v^\prime}^3\Big].
\end{eqnarray*}
Taking in consideration (\ref{q12}), the above equation can be written as
\begin{equation}\label{q15}
\alpha \cdot {\bf N}=\frac{\mu}{\kappa}(EG-F^2)\Big[(u^\prime v^{\prime \prime}-v^\prime u^{\prime \prime})+{u^\prime}^3\Gamma_{11}^2-{v^\prime}^3 \Gamma_{22}^1 +2{u^\prime}^2v^\prime \Gamma_{12}^2 +u^\prime {v^\prime}^2 \Gamma_{22}^2+({u^\prime}^2v^\prime+2u^\prime {v^\prime}^2)\Gamma_{12}^1\Big].
\end{equation}
In view of (\ref{z2.8}), (\ref{q13}) and (\ref{q15}), we get
\begin{equation}
\overline{\alpha} \cdot \overline{\bf N}= \lambda^4 \alpha \cdot {\bf N}+ \frac{\mu}{\kappa} \lambda^4 W^2\Big[{u^\prime}^3\varepsilon_{11}^2-{v^\prime}^3 \varepsilon_{22}^1 +2 {u^\prime}^2 v^\prime \varepsilon_{12}^2 +u^\prime {v^\prime}^2 \varepsilon_{22}^2+({u^\prime}^2v^\prime +2u^\prime {v^\prime}^2)\varepsilon_{12}^1   \Big].
\end{equation}
This proves the claim.
\end{proof}
\begin{corollary}
Let ${\bf M}$ and $\overline{\bf M}$ be homothetic smooth surfaces and $\alpha(s)$ be a rectifying curve on ${\bf M}$. Then the components of $\alpha(s)$ along the surface normal are also homothetic.
\end{corollary}
\begin{proof}
Let $J: {\bf M} \rightarrow \overline{\bf M}$ be a homothetic map with $\lambda(u,v)=c$, where $c$ is a non-zero, non-unit constant. Then, in view of (\ref{q8}), (\ref{q9}) and (\ref{q16}), the claim is straightforward.
\end{proof}
\begin{corollary}\cite{8}
Let ${\bf M}$ and $\overline{\bf M}$ be isometric smooth surfaces and $\alpha(s)$ be a rectifying curve on ${\bf M}$. Then the component of $\alpha(s)$ along the surface normal is invariant under such isometry. 
\end{corollary}
\begin{proof}
Let $J: {\bf M} \rightarrow \overline{\bf M}$ be an isometry. Then the dilation factor of conformality is $\lambda=1$. Therefore, from (\ref{q8}), we have
\begin{equation}
\overline{\alpha} \cdot \overline{\bf N}=[\alpha \cdot {\bf N}+ h(E,G,F,\lambda)],
\end{equation}
where $h$ is given by (\ref{q9}). From (\ref{q16}), it is straightforward to check $h\equiv 0,$ which proves our claim. 
\end{proof}
\begin{corollary}\label{crr} The Christoffel symbols are invariant under isometry. 
\end{corollary}
\begin{proof}
Let $J: {\bf M} \rightarrow \overline{\bf M}$ be an isometry. Then the dilation factor of conformality is $\lambda=1$. From (\ref{q13}) and (\ref{q16}), it is straightforward to check $\overline{\Gamma}_{ij}^k=\Gamma_{ij}^k,(i,j,k=1,2).$ In other words any quantity depending only on Christoffel symbols is invariant under isometry. 
\end{proof}
\begin{theorem}
Let ${\bf M}$ and $\overline{\bf M}$ be two conformal smooth surfaces and $\alpha(s)$ be a rectifying curve on ${\bf M}$. Then for the component of $\alpha(s)$ along any tangent vector to the surface, we have
\begin{equation}
\overline{\alpha}\cdot\overline{{\bf T}}-\alpha\cdot{\bf T}=\xi(s)(\lambda^2-1)(aEu'+aFv'+bFu'+bGv')+\frac{\mu}{\kappa}(\overline{\kappa}_n-\kappa_n),
\end{equation}
where ${\bf T}=a\varphi_u+b\varphi_v$ is any tangent vector to the surface $\bf M$ at $\alpha(s)$.
\end{theorem}
\begin{proof}
From $(\ref{2.1})$, we see that
\begin{equation*}
\alpha\cdot\varphi_u=\xi(s)(Eu'+Fv')+\frac{\mu}{\kappa}(u'^2v'L+2u'v'^2M+v'^3N).
\end{equation*}
Since ${\bf M}$ and $\overline{\bf M}$ are conformal smooth surfaces, we have
\begin{equation*}
\overline{\alpha}\cdot\overline{\varphi}_u=\xi(s)\lambda^2(Eu'+Fv')+\frac{\mu}{\kappa}v'\overline{\kappa}_n.
\end{equation*}
Therefore
\begin{equation}\label{T1}
\overline{\alpha}\cdot\overline{\varphi}_u-\alpha\cdot\varphi_u=\xi(s)(\lambda^2-1)(Eu'+Fv')+\frac{\mu}{\kappa}v'(\overline{\kappa}_n-\kappa_n).
\end{equation}
Similarly we take the component of $\alpha$ along $\varphi_v$ and obtain the following relation
\begin{equation}\label{T2}
\overline{\alpha}\cdot\overline{\varphi}_v-\alpha\cdot\varphi_v=\xi(s)(\lambda^2-1)(Fu'+Gv')+\frac{\mu}{\kappa}u'(\overline{\kappa}_n-\kappa_n).
\end{equation}
Now with the help of $(\ref{T1})$ and $(\ref{T2})$ we get
\begin{eqnarray}
\nonumber
\overline{\alpha}\cdot\overline{{\bf T}}-\alpha\cdot{\bf T}&=&a(\overline{\alpha}\cdot\overline{\varphi}_u-\alpha\cdot\varphi_u)+b(\overline{\alpha}\cdot\overline{\varphi}_v-\alpha\cdot\varphi_v)\\
\nonumber
&=&\xi(s)(\lambda^2-1)(aEu'+aFv'+bFu'+bGv')+\frac{\mu}{\kappa}(\overline{\kappa}_n-\kappa_n).
\end{eqnarray}
\end{proof}

\begin{theorem}
Let $J$ be a conformal map between two smooth surfaces ${\bf M}$ and $\overline{\bf M}$ and let $\alpha(s)$ be a rectifying curve on ${\bf M}$ such that $\overline{\alpha}(s)=J\circ \alpha(s)$ be a rectifying curve on $\overline{\bf M}$. Then the normal curvature can not be conformaly invariant and the deviation is given by
\begin{equation*}
\overline{\kappa}_n-\kappa_n=\frac{1}{\lambda^4 W^2}\Big[{u^\prime}^2(\overline{\varphi}_{uu}-\lambda^6 \varphi_{uu})+2u^\prime v^\prime (\overline{\varphi}_{uv}-\lambda^6 \varphi_{uv})+{v^\prime}^2(\overline{\varphi}_{vv}-\lambda^6 \varphi_{vv})\Big].
\end{equation*}.
\end{theorem}
\begin{proof}
Let $J:{\bf M}\rightarrow \overline{\bf M}$ be a conformal map and  $\varphi(u,v)$, $\overline{\varphi}=J\circ \alpha$ be the surface patches of $\bf M$ and $\overline{\bf M},$ respectively with at least second order derivatives being non-zero. Also let $\alpha(s)=\xi(s)\vec{t}(s)+ \mu(s)\vec{b}(s)$ be a rectifying curve on $\bf M$. Then
\begin{equation*}
\alpha^\prime(s)=\xi^\prime(s)\vec{t}(s)+\xi(s)\kappa(s)\vec{n}(s)+\mu^\prime(s)\vec{b}(s)-\mu(s)\tau(s)\vec{n}(s).
\end{equation*}
Chen \cite{2} obtained the conditions for the rectifying curve as:
\begin{equation*}
\xi^\prime(s)=1,\quad \xi(s)\kappa(s)=\mu(s)\tau(s),\quad \mu^\prime(s)=0.
\end{equation*}
Thus
$$\alpha^{\prime\prime}(s)=\kappa(s)\vec{n}(s).$$
Now
$$\kappa_n=\alpha^{\prime\prime}(s)\cdot{\bf N}=[\varphi_uu^{\prime\prime}+\varphi_vv^{\prime\prime}+\varphi_{uu}{u^\prime}^2+2 \varphi_{uv}u^\prime v^\prime+\varphi_{vv}{v^\prime}^2]\cdot {\bf N}$$
or 
$$\kappa_n= {u^\prime}^2 L+{v^\prime}^2 N+2u^\prime v^\prime M,$$
where $L,$ $M$ and $N$ are the coefficients of second fundamental form. In the Monge patch form, these coefficients are given by
\begin{equation}\label{q20}
L=\frac{\varphi_{uu}}{1+\varphi_u^2+\varphi_v^2(=W^2)},\quad M=\frac{\varphi_{uv}}{1+\varphi_u^2+\varphi_v^2},\quad N=\frac{\varphi_{vv}}{1+\varphi_u^2+\varphi_v^2}. 
\end{equation}
If $\overline{\alpha}(s)$ is a rectifying curve on $\overline{\bf M}$, then we have
\begin{equation*}
\overline{\kappa}_n-\lambda^2\kappa_n={u^\prime}^2(\overline{L}-\lambda^2 L)+{v^\prime}^2(\overline{N}- \lambda^2N)+2u^\prime v^\prime(\overline{M}- \lambda^2 M).
\end{equation*}
Using (\ref{q20}), it is easy to see that:
\begin{equation}\label{q21}
\overline{\kappa}_n-\lambda^2 \kappa_n=\frac{1}{\lambda^4 W^2}\Big[{u^\prime}^2(\overline{\varphi}_{uu}-\lambda^6 \varphi_{uu})+2u^\prime v^\prime (\overline{\varphi}_{uv}-\lambda^6 \varphi_{uv})+{v^\prime}^2(\overline{\varphi}_{vv}-\lambda^6 \varphi_{vv})\Big].
\end{equation}
Now since $u,v$ are two independent variables, from (\ref{q21}), we see that $\overline{\kappa}_n=\kappa_n$ if and only if 
\begin{equation}\label{q22}
\overline{\varphi}_{uu}-\lambda^6 \varphi_{uu}=0,\quad \overline{\varphi}_{uv}-\lambda^6 \varphi_{uv}=0,\quad \overline{\varphi}_{vv}-\lambda^6 \varphi_{vv}=0.
\end{equation}
In view of (\ref{q3.7}), it follows that the identities in (\ref{q22}) holds if and only if $\lambda=0,$ which is not the case. This shows that $\kappa_n$ is never conformaly invariant.
\end{proof}
\begin{corollary}
Let $J:{\bf M}\rightarrow \overline{\bf M}$ be a homothetic map. Then the normal curvature of the rectifying curve is not homothetic invariant under $J$ and the deviation is given by (\ref{q21}) while substituting $\lambda=c$.
\end{corollary}
\begin{corollary}
Let $J:{\bf M}\rightarrow \overline{\bf M}$ be an isometry, then the normal curvature of the rectifying curve is not invariant under $J$ and the deviation is given by (\ref{q21}) while substituting $\lambda=1$.
\end{corollary}
\begin{theorem}
Let $J$ be a conformal map between two smooth surfaces ${\bf M}$ and $\overline{\bf M}$ and let $\alpha(s)$ be a rectifying curve on ${\bf M}$ such that $\overline{\alpha}(s)=J\circ \alpha(s)$ be a rectifying curve on $\overline{\bf M}$. Then the geodesic curvature of $\alpha(s)$ is not conformally invariant and the deviation is given by
\begin{equation*}
\overline{\kappa}_g-\lambda^2\kappa_g=\Big[\epsilon_{11}^2{u^\prime}^3+(2\epsilon_{12}^2-\epsilon_{11}^1){u^\prime}^2 v^\prime +(\epsilon_{22}^2-2\epsilon_{12}^1)u^\prime {v^\prime}^2-\epsilon_{22}^1{v^\prime}^3\Big] \sqrt{EG-F^2}.
\end{equation*}
where $\epsilon_{ij}^k(i,j,k=1,2)$ are given by (\ref{q16}).
\end{theorem}
\begin{proof}
Let $\alpha$ be a rectifying curve on a parametric surface ${\bf M}$ and $\overline{\alpha}(s)$ be a rectifying curve on $\overline{\bf M}$. In \cite{8} Shaikh and Ghosh showed that:
\begin{eqnarray*}
\kappa_g&=&(v^\prime u^{\prime\prime}-u^\prime v^{\prime\prime})(F^2-EG)+\frac{1}{2}{u^\prime}^3(2F_uE-E_vE-E_uF)\\
&&+\frac{1}{2}{u^\prime}^2v^\prime(2F_uF-E_vF-E_uG)+{u^\prime}^2v^\prime(G_uE-E_vF)+u^\prime{v^\prime
}^2(G_uF-E_vG)\\
&&+\frac{1}{2}u^\prime{v^\prime}^2(G_vE-2F_vF+G_uF)+\frac{1}{2}{v^\prime}^3(G_vF-2F_vG+G_uG).
\end{eqnarray*}
Using (\ref{q12}), the above equation turns out to be
\begin{equation}\label{w23}
\kappa_g=\Big[\Gamma_{11}^2{u^\prime}^3+(2\Gamma_{12}^2-\Gamma_{11}^1){u^\prime}^2 v^\prime +(\Gamma_{22}^2-2\Gamma_{12}^1)u^\prime {v^\prime}^2-\Gamma_{22}^1{v^\prime}^3+u^\prime v^{\prime\prime}-u^{\prime \prime}v^\prime\Big] \sqrt{EG-F^2}.
\end{equation}
In view of (\ref{q13}) and the above equation, $\overline{\kappa}_g$ is given by
\begin{equation}\label{q23}
\overline{\kappa}_g=\lambda^2\kappa_g+\Big[\epsilon_{11}^2{u^\prime}^3+(2\epsilon_{12}^2-\epsilon_{11}^1){u^\prime}^2 v^\prime +(\epsilon_{22}^2-2\epsilon_{12}^1)u^\prime {v^\prime}^2-\epsilon_{22}^1{v^\prime}^3\Big] \sqrt{EG-F^2}.
\end{equation}
\end{proof}
\begin{corollary}
Let $J$ be a homothetic map between two smooth surfaces $\bf M$ and $\overline{\bf M}$. Then the geodesic curvature of rectifying curve is homothetic invariant under $J$.
\end{corollary}
\begin{proof}
Let us suppose $\lambda=c$. Then the proof is a direct implication of (\ref{q16}) and (\ref{q23}).
\end{proof} 
\begin{corollary}\cite{8}
Let $J$ be an isometry between two smooth surfaces $\bf M$ and $\overline{\bf M}$. Then the geodesic curvature of rectifying curve is invariant under $J$.
\end{corollary}
\begin{proof}
For isometry, we have $\lambda=1$. Thus in view of Corollary \ref{crr} and the equations (\ref{q16}), (\ref{q23}), the claim is straightforward.
\end{proof}
\section{acknowledgment}
 The third author greatly acknowledges to The University Grants Commission, Government of India for the award of Junior Research Fellow.

\end{document}